\documentclass[11pt]{amsart}
\usepackage{graphicx,color}

\theoremstyle{plain}

\theoremstyle{plain}
\newtheorem{theorem}{Theorem} [section]
\newtheorem{corollary}[theorem]{Corollary}
\newtheorem{lemma}[theorem]{Lemma}
\newtheorem{proposition}[theorem]{Proposition}

\theoremstyle{definition}

\theoremstyle{remark}

\numberwithin{theorem}{section}
\numberwithin{equation}{section}
\numberwithin{figure}{section}

\def\mean#1{\mathchoice
         {\mathop{\kern 0.2em\vrule width 0.6em height 0.69678ex depth -0.58065ex
                 \kern -0.8em \intop}\nolimits_{\kern -0.4em#1}}%
         {\mathop{\kern 0.1em\vrule width 0.5em height 0.69678ex depth -0.60387ex
                 \kern -0.6em \intop}\nolimits_{#1}}%
         {\mathop{\kern 0.1em\vrule width 0.5em height 0.69678ex
             depth -0.60387ex
                 \kern -0.6em \intop}\nolimits_{#1}}%
         {\mathop{\kern 0.1em\vrule width 0.5em height 0.69678ex depth -0.60387ex
                 \kern -0.6em \intop}\nolimits_{#1}}}

\def\N{\mathbb N}

\def\R{\mathbb R}

\def\a{\alpha}

\def\e{\varepsilon}

\def\l{\lambda}
\def\L{\Lambda}
\def\O{\Omega}

\def\U{\mathcal U}

\DeclareMathOperator{\osc}{osc}
 \DeclareMathOperator{\dist}{dist}

\title[$W^{2,1}$ regularity for solutions
of the Monge-Amp\`ere equation]{$W^{2,1}$ regularity for solutions\\
of the Monge-Amp\`ere equation}

\author[G. De Philippis]{Guido De Philippis}
\address{Scuola Normale Superiore,
p.za dei Cavalieri 7, I-56126 Pisa, Italy}
\email{guido.dephilippis@sns.it}

\author[A. Figalli]{Alessio Figalli}
\address{Department of Mathematics,
The University of Texas at Austin, 1 University Station C1200,
Austin TX 78712, USA}
\email{figalli@math.utexas.edu}


\begin{document}

\begin{abstract}
In this paper we prove that a strictly convex Alexandrov solution $u$
of the Monge-Amp\`ere equation,
with right hand side bounded away from zero and infinity, is $W^{2,1}_{\rm loc}$.
This is obtained by showing higher integrability a-priori estimates for $D^2u$,
namely $D^2 u \in L \log^k L$ for any $k \in \N$.
\end{abstract}

\maketitle

\section{Introduction}

Let $\Omega \subset \R^n$ be a bounded convex domain,
and $u:\overline\Omega \to \R$ a continuous convex function solving the
Monge-Amp\`ere equation
\begin{equation}
\label{MA}
\begin{cases}
\det D^2 u=f \quad &\text{in $\Omega$}\\
u=0 &\text{on $\partial \Omega$}
\end{cases}
\end{equation}
in the Alexandrov sense.
Whenever $f:\Omega \to \R^+$ is positive and smooth, solutions to such equation
are smooth as well \cite{U}.
However, for several applications it is important
to understand the regularity of $u$ when $f$ does not enjoy any regularity.
More precisely, we want to investigate the properties of $u$
under the only assumption that there exist positive constants $\l,\L>0$
such that $0<\l \leq f \leq \L$ inside $\Omega$.\\

Under such assumptions on $f$, Caffarelli proved that
solutions are strictly convex and $C^{1,\alpha}$ \cite{CA1,CA2}.
In particular, in his works \cite{CA1,CA3}
he could deduce the two following corollaries:
\begin{enumerate}
\item[$\bullet$] \textit{Minkowski Problem \cite{CA1,CA2}:} Let $\Gamma\subset \R^{n+1}$ be a bounded convex set,
and assume that its Gauss curvature $\mathcal G$ satisfies 
$0 < \l \leq \mathcal G \leq \L$. Then $\Gamma$ is strictly convex,
and $\partial\Gamma$ is $C^{1,\alpha}$. 
\item[$\bullet$] \textit{Optimal Transport \cite{CA3}:} Let 
$\O_1,\O_2\subset \R^n$ be two bounded open sets, and $f_1,f_2$ two probability densities
such that $0<\l \leq f_1,f_2\leq
\Lambda$ inside $\O_1$ and $\O_2$ respectively.
Let $u:\R^n \to \R$ be a convex function such that $(\nabla u)_\# f_1=f_2$,
and assume that $\O_2$ is convex.
Then $u \in C^{1,\alpha}_{\rm loc}(\O_1)$.
(This corresponds to say that optimal transport maps
with convex targets are H\"older continuous, see \cite{brenier1,brenier2,CA3}.) 
\end{enumerate}

In this paper we want to investigate the Sobolev regularity of $u$.
In \cite{CAW2p} Caffarelli showed that for any $p>1$ there exists $\e=\e(p)>0$ such that
if $|f-1| \leq \e$, then $u \in W^{2,p}_{\rm loc}(\Omega)$.
Few years later \cite{wang}, Wang constructed examples of solutions to \eqref{MA}, with
$0<\lambda \le f \le \Lambda$ and $\L/\l$ large, which are not $W^{2,p}$ for some $p>1$.
Moreover, by taking $\L/\l$ large enough, $p$ can be chosen as close to $1$ as desired.

These results left open the question of whether solutions to \eqref{MA} with $0<\l \leq f \leq \L$
belong to $W^{2,1}_{\rm loc}(\Omega)$. (This question is also raised as an open problem in
\cite[Section 7.6]{ambrosioCetraro} in connection with the semigeostrophic equations.)

Let us observe that, since $u$ is convex, its Hessian exists in the sense of distributions
and defines a (locally finite) non-negative measure. 
Moreover, the $C^{1,\alpha}$ regularity result of Caffarelli already
shows that $\nabla u$ is a $BV$ map whose distributional derivative has no jump part.
However, to prove that $u \in W^{2,1}_{\rm loc}(\Omega)$ one still needs to rule out the Cantor part.

In these last years, several attempts have been made both to prove such a result
and to construct a counterexample.
In particular, in a recent paper \cite{ADK} the authors connect the Sobolev regularity of $u$
to a differential inclusion in the space of symmetric matrices.\\

In this paper we finally give a positive answer to this problem
by directly working at the level of
the Monge-Amp\`ere equation.
Indeed, we prove not only that $u \in W^{2,1}_{\rm loc}(\Omega)$, but we can actually
show a higher integrability estimate for $D^2 u$. 
Here is our result:
\begin{theorem}\label{reg}
Let $\Omega\subset \R^n$ be a bounded convex domain, and  $u:\overline\Omega \to \R$ be an Alexandrov solution of
\eqref{MA}
with  $0<\lambda \le f \le \Lambda$. Then, for any $\Omega'\subset\subset \Omega$ and $k \in \N\cup \{0\}$,
there exists a constant $C=C(k,n,\l,\L,\Omega,\Omega')>0$ such that
\begin{equation}\label{ulogu}
\int_{\Omega'}\|D^2 u\| \log^k\big(2+ \|D^2 u\|\big) \leq C.
\end{equation}
In particular $u \in W^{2,1}_{\rm loc}(\Omega)$.
\end{theorem}

As a corollary we obtain the following Sobolev regularity result for
optimal transport maps (of course,
our theorem applies also to the Minkowski problem, yielding $W^{2,1}$ regularity of $\partial \Gamma$):
\begin{corollary}
Let $\O_1,\O_2\subset \R^n$ be two bounded domains, and $f_1,f_2$ two probability densities
such that $0<\l \leq f_1,f_2\leq
\Lambda$ inside $\O_1$ and $\O_2$ respectively.
Let $T=\nabla u:\O_1\to \O_2$ be the (unique) optimal transport map for the quadratic cost
sending $f_1$ onto $f_2$,
and assume that $\O_2$ is convex.
Then $T \in W^{1,1}_{\rm loc}(\O_1)$.
\end{corollary}

Let us observe that, if in the above corollary one removes the convexity assumption on $\Omega_2$,
by the results in \cite{figkim} we can still deduce that there exists a closed set $\Sigma$ of Lebesgue
measure zero such that $T \in W^{1,1}_{\rm loc}(\O_1\setminus \Sigma)$.
Moreover, combining Theorem \ref{reg} (see also Theorem \ref{main} below) with some recent results of Savin \cite{Savinboundary},
we can apply verbatim his argument in \cite{Savinglobal} to obtain global
$W^{2,1}$ regularity:
more precisely, under our assumptions on $f$,
Theorem 1.2 and Corollary 1.4 in \cite{Savinglobal} hold replacing the $L^p$
norm of $D^2u$ with its $L\log^kL$ norm.\\

We now make some comments on Theorem \ref{reg}.

First of all we remark that, for his $C^{1,\alpha}$ regularity result, Caffarelli did not need
to assume that $\det D^2u$ is bounded from above and below (as in our case),
but only that it defines
a doubling measure (see \cite{CA2} or \cite[Section 3.1]{G} for a precise definition).
However, even for $n=1$ there are doubling measures which are singular with respect to
the Lebesgue measure (see for instance \cite[Chapter 1, Section 8.8(a)]{St} and \cite[Chapter 5, $\S$ 7]{zyg}),
so $u''=\mu$ (the $1$-d version of Monge-Amp\`ere) with $\mu$ doubling cannot imply $W^{2,1}$-regularity.
Moreover, in connection with what was mentioned before,
Wang's counterexamples \cite{wang} show that our result is almost optimal
(still, his examples do not exclude that
 $u$ may be $W^{2,p}$ for some $p=p(n,\l,\L)>1$).\\

The paper is structured as follows: in Section \ref{sec:notation} we introduce
the notation and collect
some preliminary results.
Then in Section \ref{sec:proof} we prove Theorem \ref{reg}. As we will show, Theorem \ref{reg}
is an easy consequence of Theorem \ref{main}, which is the main result of this paper.
\\

\textit{Acknowledgments:} We thank Luigi Ambrosio and Luis Caffarelli
for several useful discussions  about this problem.
We also thank Diego Maldonado for a careful reading of a preliminary version of this paper
and for several useful comments.
The authors acknowledge the hospitality at the Mathematisches Forschungsinstitut
Oberwolfach during the 2011 workshop ``Partial Differential Equations'',
where part of the present work has been done. AF has been partially supported by NSF Grant DMS-0969962.
Both authors acknowledge the support of the ERC ADG Grant GeMeThNES.

\section{Properties of solutions of the Monge Amp\`ere equation and of their sections}
\label{sec:notation}

In this section we recall some basic facts on Alexandrov solutions of the Monge Amp\`ere equation
and the geometric properties of their sections. We refer the reader to \cite{G}
for a detailed exposition on these subjects.\\

Given a Radon measure $\mu$ on $\R^n$, and a bounded convex domain $\Omega \subset \R^n$,
we say that a convex function $u:\Omega \to \R$ is an Alexandrov solution of the Monge-Amp\`ere equation
\[
\begin{cases}
\det D^2 u=\mu \quad &\text{in }\Omega\\
u=0 &\text{on $\partial \Omega$}
\end{cases}
\]
if for any Borel set $B\subset \Omega $ it holds
\[
\Big|\bigcup_{x \in B} \partial u(x)\Big|=\mu(B),
\]
where $\partial u(x)$ denotes the subdifferential of $u$ at $x$.
(Here and in the sequel, $|E|$ denotes the Lebesgue measure of a set $E$.)
As mentioned in the introduction, Caffarelli proved that 
if $\mu$ is a doubling measure inside $\Omega$
(in particular, if $\mu=f\,dx$ with $0<\lambda \le f \le \Lambda$
inside $\Omega$), then $u$ is strictly convex and $u \in C^{1,\alpha}_{\rm loc}(\Omega)$ \cite{CA1,CA2}\footnote{More precisely,
Caffarelli proved the following: (1) if $\ell:\R^n\to \R$ is a supporting linear function for $u$
and $\{u=\ell\}$ is not a point,
then all extremal points of the convex set  $\{u=\ell\}$ are contained in $\partial\Omega$; (2)
if $u$ is strictly convex inside $\Omega$, then $u \in C^{1,\alpha}_{\rm loc}(\Omega)$.
In our situation, since $u=0$ on $\partial\O$, (1) forces the set
$\{u=\ell\}$ to be reduced to a point for any
supporting linear function $\ell$ (since $u\not\equiv 0$),
and so (2) implies that $u \in C^{1,\alpha}_{\rm loc}(\Omega)$.}

A key role in the proof of the previous result is played by the
\emph{sections} of $u$, which play for the Monge-Amp\`ere equation
the same role that balls play for an uniformly elliptic equation.
We recall some important definitions and properties which we will use
in the proof of Theorem \ref{reg}.\\

Given $u:\Omega \to \R$ a convex function, for any point
$x$ in $\Omega$, $p \in \partial u (x)$, and  $t \geq 0$, we define
the section centered at $x$ with height $t$ (with respect to $p$) as
\begin{equation}\label{sec}
S(x,p,t):=\big\{y\in \Omega\,:\ u(y)\le u(x)+p\cdot(y-x)+t\big\}. 
\end{equation}
When $u $ is continuously differentiable 
$\partial u(x)$ reduces to $\{\nabla u(x)\}$, and in this case we will simply write
$S(x,t)$ for $S(x,\nabla u(x), t)$. Moreover, given $\tau>0$,
we will use the notation  $\tau S(x,p,t)$ to denote the dilation of $S(x,p,t)$ by a factor $\tau$ 
with respect to $x$,\footnote{
We remark that one could also consider dilations with respect to the center of mass
of the sections, and the geometric properties described  in Proposition \ref{secprop} are true in both cases.
However, for our estimates, the choice of dilating with respect to $x$
is more convenient.
} that is
\begin{equation}
\label{eq:dilation}
\tau S(x,p,t):=\left\{y \in \R^n \,: \, x+\frac{y-x}{\tau} \in S(x,p,t)\right\}.
\end{equation}
\\

We say that an open bounded convex set $Z \subset \R^n$ is \emph{normalized} if
\[
B(0,1)\subset Z \subset B(0,n).
\] 
By John's Lemma \cite{John48},
for  every open bounded convex set there exists an (invertible) orientation preserving
affine transformation $T:\R^n\to \R^n$ such that $T(Z)$ is normalized. In particular
\begin{equation}\label{det T}
\frac {\omega_n} {|Z|} \le \det T\le \frac {n^{n} \omega_n} {|Z|},\qquad \text{where }\omega_n:=|B(0,1)|. 
\end{equation}

Notice that in the sequel we are not going to notationally distinguish between an
affine transformation and its linear part, since
it will always be clear to what we are referring to.
In particular, we will use the notation
\begin{equation}
\label{eq:norm}
\|T \|:= \sup_{|v| =1} |Av|,\qquad Tx=Ax+b.
\end{equation}
One useful property which we will use is the following identity:
if we denote by $T^*$ the adjoint of $T$, then 
\begin{equation}
\label{eq:norms}
\|T^*T \|=\|T^*\|\|T \|.
\end{equation}
(This can be easily proved using the polar decomposition of matrices.)
\\

Whenever $u$ is a strictly convex solution of \eqref{MA} with $ 0 < \l \leq f \leq \L$
(in particular $u \in C^{1,\alpha}$),
for any $x \in \Omega$
one can choose $t>0$ sufficiently small so that $S(x,t) \subset \subset \Omega$.
Then, if $T$ is the affine transformation which normalizes $S(x,t)$, the function  
\begin{equation}\label{vnor}
v(z):= (\det T)^{2/n}\left[ u(T^{-1}z)-u(x) - \nabla u(x) \cdot (T^{-1}z -x)-t \right]
\end{equation}
solves
\begin{equation}\label{MAbdry}
 \begin{cases}
 \l \leq \det D^2 v \leq \L\quad &\text{in $Z$},\\
 v=0 &\text{on $\partial Z$},
 \end{cases}
\end{equation}
with $Z:=T(S(x,t))$ renormalized. We are going to call $v$ a \emph{normalized solution}.

Whenever $v$ is a normalized solution, it easily follows from 
Alexandrov maximum principle that there exist two constants $c_1,c_2>0$,
depending only on $n,\lambda, \Lambda$, such that
\begin{equation}
\label{eq:Alex}
c_1\le \Bigl|\inf_{Z} v\Bigr|  \le c_2,
\end{equation}
see \cite[Proposition 3.2.3]{G}.
In the sequel we are going to call \emph{universal}
any constant which depends only on $n,\lambda, \Lambda$.\\

As shown in \cite{CA2} and \cite{GH} (see also \cite[Chapter 3]{G}),
sections of solution of \eqref{MA} satisfy
strong geometric properties.
We briefly recall here the ones we are going to use:

\begin{proposition}\label{secprop}
Let $u$ be a strictly convex Alexandrov solution of \eqref{MA} with
$ 0 < \l \leq f \leq \L$. Then, for any $\Omega' \subset \subset \Omega''\subset \Omega$,
there exists a positive constant
$\rho=\rho(n,\lambda, \Lambda,\Omega',\Omega'')$ such that
the following properties hold:
\begin{itemize}
\item[(i)] $S(x,t) \subset \Omega''$ for any $x\in \Omega'$, 
$0 \leq t\le 2\rho$.
\item[(ii)] For all $\tau \in (0,1)$ there exists $\beta=\beta(\tau,n,\l,\L) \in (0,1)$
such that $\tau S(x,t) \subset S(x,\tau t)\subset \beta S(x,t)$ for any $x\in\Omega'$,
$0 \leq t\le 2\rho$.
\item[(iii)] There exists a universal constant $\theta>1$
such that, if $S(x,t)\cap S(y,t) \ne \emptyset$, then $S(y,t) \subset S(x,\theta t)$ 
for any $x,y\in\Omega'$, 
$0 \leq t\le 2\rho/\theta$.
\item[(iv)] $\cap_{0 <t \leq \rho} S(x,t)=\{x\}$.
\end{itemize} 
\end{proposition}

The previous properties play a key
role, since they allow to ``use sections as they were balls''.
In particular the following covering lemma holds, see \cite[Lemma 1]{CAGU}
(we recall that  $\chi_E$ denotes the characteristic function of a set $E$):
\begin{proposition}\label{covering}
Let $u$ be a strictly convex Alexandrov solution of \eqref{MA} with
$ 0 < \l \leq f \leq \L$.
Let $\Omega' \subset \subset \Omega''\subset \Omega$,
and let $A \subset  \Omega'$.
Suppose there exists a family of section $\mathcal F=\{ S(x,t_x)\}_{x \in A}$,
with 
$t_x \le \rho$ for every $x \in A$ (here $\rho$ is as in Proposition \ref{secprop}).
Then there exists a countable subfamily
of $\mathcal G=\{S(x_k, t_{x_k} )\}_{k \in \N}$, with the following properties:
\begin{enumerate}
\item[(i)] $A \subset \bigcup_{k \in \N} S(x_k , t_{x_k} )$,
\item [(ii)] there exist two universal constants $\e_0$ and $K$
such that for every $\e\le \e_0$ it holds
\[
\sum_{k \in \N} \chi_{S(x_k ,(1-\e) t_{x_k} )}(x)\le K|\log \e| \qquad \forall \,x \in \Omega''.
\]  
\end{enumerate}
\end{proposition}

By Proposition \ref{secprop}
the sections satisfy the list of axioms in \cite[Section 1.1]{St},
so several classical theorems in real analysis hold
using sections in place of Euclidean balls (see \cite[Chapter 1]{St}).
In particular, an important tool we are going to use is the maximal operator
defined through sections.
In order to introduce it, we assume here that 
$u$ is a $C^2$ solution. As we will discuss in the next section,
this can be done without loss of generality as long as all the constants involved in the bounds
are universal.

For any $x \in\Omega' \subset \subset \Omega''\subset \subset \Omega$
and $\rho$ is as in Proposition \ref{secprop}, we define 
\begin{equation}
\label{eq:def max fct}
M_{\Omega',\Omega''}(x):= \sup_{0<t<\rho}\quad \mean{S(x,t)} \|D^2 u(y)\|\,dy,
\end{equation}
where $\|D^2 u(y)\|$ denotes the operator norm of the matrix $D^2u(y)$,
and $\rho$ is as in Proposition \ref{secprop}.
By \cite[Chapter 1, Section 4, Theorem 2]{St} and \cite[Chapter 1, Section 8.14]{St},
the following key property holds:
there exist universal constants $C',C''>0$ such that,
for any $\a\ge \a_0$,
\begin{equation}\label{maximalineq}
 \int\limits_{\{|D^2u\| \ge \a\}\cap \Omega'} \|D^2 u\|
\le C' \a \bigl|{\{|M_{\Omega',\Omega''}(x) \ge C'' \a\}\cap \Omega''}\bigr|.
\end{equation}
Here $\alpha_0 $ is a sufficiently large constant which depends only on $\mean{\Omega'}\|D^2 u\|$ and $\rho$.

\section{Proof of Theorem \ref{reg}}\label{sec:proof}
By standard approximation arguments, it suffices to prove \eqref{ulogu} when
$u \in C^2(\Omega)$.

Let us remark that the proof of \eqref{ulogu} for $k=0$ is elementary: indeed, this follows
from $\|D^2 u\| \leq \Delta u$ (since $u$ is convex) and
a universal interior bound for the gradient of $u$
(see for instance \cite[Lemma 3.2.1 and Proposition 3.2.3]{G} or \eqref{eq:v 2}-\eqref{eq:v 3} below).

Hence, performing an induction on $k$ and using
a standard covering argument (briefly sketched below),
it suffices to prove the following result (recall the notation \eqref{eq:dilation} for the dilation of a section):

\begin{theorem}\label{main}
Let $\U \subset \R^n$ be a normalized convex set,
and $u:\U \to \R$ a $C^2$ convex solution of
\begin{equation}\label{MAnorm}
 \begin{cases}
0<\l \leq  \det D^2 u \leq \L\quad &\text{in $\U$},\\
 u=0 &\text{on $\partial \U$}.\\
 \end{cases}
\end{equation}
Then for any  $k \in \mathbb N \cup \{0\}$
there exists a constant $C=C(k,n,\lambda,\Lambda)$ such that
\[
\int_{ \U/2} \|D^2 u\|\log^{k+1}\big(2+ \|D^2 u\|\big) \le C \int_{3\U/4} \|D^2 u\|\log^{k}\big(2+ \|D^2 u\|\big).
\]
\end{theorem}

\medskip
\begin{proof}[Proof of \eqref{ulogu} using Theorem \ref{main}]
We want to show that, if \eqref{ulogu} holds for some $k \in \N \cup \{0\}$,
then it also holds for $k+1$.
Since the following argument is standard, we just give a sketch of the proof.

Given $\Omega' \subset \subset \Omega$, 
fix $\rho$ as in Proposition \ref{secprop} with $\Omega''=\Omega$, and 
consider the covering of $\Omega'$ given by $\{S(x,\rho)\}_{x \in \Omega'}$.
By \cite[Theorem 3.3.8]{G} (see also \cite[Condition A]{CAGU}) all sections
$\{S(x,\rho)\}_{x \in \Omega'}$ have comparable shapes, and such shapes are also comparable to the one of $\Omega$
(since $\Omega$ is also a section for $u$); more precisely, there exist positive constants $r_1,r_2$, depending only on $n,\l,\L,\rho,\Omega$,
such that
\begin{equation}
\label{eq:comp}
B(x,r_1) \subset S(x,\rho)\subset B(x,r_2) \qquad \forall\, x \in \Omega'. 
\end{equation}
This implies that one
can cover $\Omega'$ with finitely
many such sections $\{S_i\}_{i=1,\ldots,N}$ (the number $N$ depending only on $r_1,r_2,\Omega'$),
and moreover the affine transformations $T_i$ normalizing them satisfy the following bounds
(which follow easily from \eqref{eq:comp} and the inclusion $B(0,1)\subset T_i(S_i)\subset B(0,n)$):
$$
\|T_i\| \leq \frac{n}{r_1},\qquad \det T_i \geq \frac{1}{r_2^n}.
$$
Hence, we can define $v_i$ as in \eqref{vnor} with $T=T_i$ and $t=\rho$,
and apply Theorem \ref{main} to each of them: by using the inductive hypothesis we have
$$
\int_{ T_i(S_i)/2} \|D^2 v_i\|\log^{k+1}\big(2+ \|D^2 v_i\|\big) \le C(k,n,\l,\L).
$$
Changing variables back and summing over $i$, we get
$$
\int_{\Omega'} \|D^2 u\|\log^k\big(2+ \|D^2 u\|\big) \leq C(k,n,\l,\L)\sum_{i=1}^N
\frac{\|T_i\|\|T_i^*\|}{(\det T_i)^{1+2/n}}
\log\left(2+\frac{\|T_i\|\|T_i^*\|}{(\det T_i)^{2/n}} \right).
$$
Recalling that we have uniform bounds on $N$ and on $T_i$, this concludes the proof.
\end{proof}

We now focus on the proof of Theorem \ref{main}.
We begin by showing that the average of $\|D^2 u\|$ over a section
is controlled by the size of the ``normalizing affine transformation''.

\begin{lemma}\label{hessmean}
Let $u$ solve \eqref{MAnorm}, fix $x \in \U/2$, and let $t>0$ be such that $S(x,2t) \subset 3\U/4$.
Let $T$ be the affine map which normalizes $S(x,t)$.
Then there exists a positive universal constant $C_1$ such that
\begin{equation}\label{average}
 \frac{\|T\|\|T^*\|}{(\det T)^{2/n}} \geq C_1\mean{S(x,t)} \|D^2 u\|. 
\end{equation}
\end{lemma}

\begin{proof}
Let us consider $v: T(S(x,2t)) \to \R$, with $v$ is defined as in \eqref{vnor}.
We notice that 
\begin{equation}
\label{eq:v 1}
D^2 v(z)=(\det T)^{2/n}\left[(T^{-1})^* D^2 u(T^{-1} z)T^{-1}\right], 
\end{equation}
and
\begin{equation}
\label{eq:v bis}
 \begin{cases}
 \l \leq \det D^2 v \leq \L\quad &\text{in $T(S(x,2t))$},\\
 v=t &\text{on $\partial \bigl(T(S(x,2t))\bigr)$}.
 \end{cases}
\end{equation}
Although the convex set $T(S(x,2t))$ is not renormalized in the sense defined before, it
is almost so: indeed, since $T$ normalizes $S(x,t)$ we have $Tz \in B(0,n)$ for any $z \in S(x,t)$.
Recalling that $2S(x,t)$ denotes the dilation of $S(x,t)$ with respect to $x$ (see \eqref{eq:dilation}), we get
$$
T\left(x+\frac{y-x}2\right) \in B(0,n) \qquad \forall\, y \in 2S(x,t),
$$
which is equivalent to
$$
Ty + Tx \in B(0,2n)\qquad \forall\, y \in 2S(x,t).
$$
Since $Tx \in B(0,n)$ this implies that $T(2S(x,t)) \subset B(0,3n)$,
which together with the fact that
$S(x,t) \subset S(x,2t) \subset 2S(x,t)$ (by convexity of $u$)
gives
\begin{equation}
\label{eq:normalize 2S} 
B(0,1) \subset T(S(x,2t))\subset B(0,3n).
\end{equation}
Hence, it follows from \eqref{eq:v bis} and \cite[Proposition 3.2.3]{G} that
\begin{equation}
\label{eq:v 2}
\osc_{T(S(x,2t))} v = \left|\inf_{T(S(x,2t))} (v-t)\right| \le c',
\end{equation}
with $c'$ universal.

Since $v$ is convex, by
\cite[Lemma 3.2.1]{G}, \eqref{eq:normalize 2S},
\eqref{eq:v 2}, and Proposition \ref{secprop}(ii) applied with $\tau=1/2$, we also get
\begin{equation}
\label{eq:v 3}
\sup_{T(S(x,t))} |\nabla v| \leq \sup_{\beta T(S(x,2t))} |\nabla v| \le
\frac{\osc_{T(S(x,2t))} v}{\dist\big(\beta T(S(x,2t)) ,\,\partial \bigl(T(S(x,2t))\bigr)\big)}
\leq
c''
\end{equation}
for some universal constant $c''$.
Moreover, since
$T(S(x,t))$ is a normalized convex set, it holds 
\begin{equation}
\label{eq:v 4}
\omega_n \leq |T(S(x,t))| = \det T |S(x,t)|,
\qquad \mathcal H ^{n-1}\bigl(\partial T(S(x,t))\bigr)\le c(n),
\end{equation}
where $c(n)$ is a dimensional constant (recall that $\omega_n=|B(0,1)|$).
Finally, using again the convexity of $v$, the estimate
\begin{equation}
\label{eq:v 5}
\|D^2v(y)\| \leq \Delta v(y)
\end{equation}
holds (recall that $\|D^2 u(y)\|$ denotes the operator norm of $D^2u(y)$).

Hence,
by \eqref{eq:v 1}, \eqref{eq:v 4}, \eqref{eq:v 5} and \eqref{eq:v 3}, we get
\[
\begin{split}
\mean{S(x,t)} \|D^2 u(y)\|\,dy &= \frac{1}{(\det T)^{2/n}}\mean{S(x,t)} \|T^*D^2 v(Ty)T\|\,dy\\
&\le  \frac{\|T^*\|\|T\|}{(\det T)^{2/n}} \frac{1}{\det T |S(x,t)|}\int_{T(S(x,t))} \|D^2 v(z)\|\,dz\\
&\le \frac{\|T^*\|\|T\|}{(\det T)^{2/n}\omega_n}\int_{T(S(x,t))} \Delta v(z) \,dz\\
&= \frac{\|T^*\|\|T\|}{(\det T)^{2/n}\omega_n}\int_{T(\partial S(x,t))} \nabla v(z)\cdot \nu\, d\mathcal H^{n-1}(z)\\
& \le \frac{c(n)\|T^*\|\|T\|}{(\det T)^{2/n}\omega_n} \sup_{T(S(x,t))} |\nabla v|\\
& \le c'' \frac{c(n)\|T^*\|\|T\|}{(\det T)^{2/n}\omega_n},
\end{split}
\]
which concludes the proof of \eqref{average}.
\end{proof}

We now show that, in every section, we can find a uniform
fraction of points where the norm of the Hessian controls the size of the
``normalizing affine transformation''.
\begin{lemma}\label{hesssupermean}
Let $u$ solve \eqref{MAnorm}, fix $x \in \U/2$, and let $t>0$ be such that $S(x,2t) \subset 3\U/4$.
Let $T$ be the affine map which normalizes $S(x,t)$.
Then there exist universal positive constants $C_2$, $C_3$, $\e_1$, with $\e_1 \in (0,1)$,
and a Borel set $A(x,t)\subset S(x,t)$, such that 
\begin{equation}\label{A}
\frac{|A(x,t)\cap S(x,(1-\e)t)|}{|S(x,t)|}\ge C_2 \qquad \forall\, 0 \leq \e \leq \e_1,
\end{equation}
and 
 \begin{equation}\label{eq: hesssupermean}
 \|D^2 u(y)\| \ge C_3 \frac{\|T\|\|T^*\|}{(\det T)^{2/n}} \qquad \forall\, y \in A(x,t).
\end{equation}
\end{lemma}

\begin{proof}
We divide the proof in two steps.\\

\noindent {\it Step one: Let $v$ be a normalized solution in $Z$ (see \eqref{MAbdry}).
Then there exist universal constants $c',c''>0$, and a Borel set $E\subset Z$, such that
$|E|\ge c'|Z| $, and $D^2 v(x) \ge c'' {\rm Id}$ for every $x \in E$.}

To see this, let us consider the paraboloid
$p(x):=c_1(|x|^2/n^{2}-1)/2$, with $c_1$ as in \eqref{eq:Alex}
(observe that, since $Z \subset B(0,n)$, $p \leq 0$ inside $Z$). Then 
\[
|\inf_{\Omega} (v-p)|\ge \frac{c_1}{2}. 
\] 
Set $w:=v-p$, and let $\Gamma_w:Z\to \R$ be a convex envelope of $w$ in $Z$, that is
$$
\Gamma_w(y):=\sup\{ \ell(y)\,:\, \ell \leq w \text{ in }Z,\,
\ell \leq 0 \text{ on }\partial Z,\, \ell \text{ affine}\}.
$$
It is well-known that $\Gamma_w$ is $C^{1,1}(Z)$,
and that $\det D^2 \Gamma_w=0$ outside the set $\{\Gamma_w=w\}\subset Z$ (see for instance
\cite[Proposition 6.6.1]{G}). Hence, by Alexandrov Maximum Principle
(see for instance \cite[Lemma 3.2.2]{G}) and from the fact that
$$
0 \leq D^2 \Gamma_w \leq D^2 w \leq D^2 v \qquad \text{a.e. on }\{\Gamma_w=w\}
$$
(in the sense of non-negative symmetric matrices), we get
\[
\begin{split}
\biggl(\frac{c_1}{2}\biggr)^n&\le \bigl|\inf_{Z} w\bigr|^n= \bigl|\inf_{Z} \Gamma_w\bigr|^n
\le C(n) \int_{\{\Gamma_w=w\}} \det D^2 \Gamma_w\\ 
&\leq  C(n) \int_{\{\Gamma_w=w\}} \det D^2 v
\le C(n) \Lambda \big|\{\Gamma_w=w\}\big|.
\end{split}
\]
This provides a universal lower bound on the measure of $E:=\{\Gamma_w=w\}$.

Moreover, since $D^2 w \geq 0$ on $E$, we obtain
$$
D^2 v \geq \frac{c_1}{n^2} {\rm Id} \qquad \text{on }E,
$$
proving the claim.\\

\noindent{\it Step two: Proof of the lemma.}
Let $S(x,t)$ and $T$ be as in the statement of the lemma, and define $v$ as in \eqref{vnor}.
Since $v$ is a normalized solution, we can apply the previous step to find a set
$E \subset Z:=T(S(x,t))$ such that $|E|\ge c'|Z|$, and $D^2 v\ge c'' {\rm Id}$ on $E$.
We define $A(x,t):= T^{-1} (E)$.

To prove \eqref{A} we observe that,
since $S(x,(1-\e)t)\supset (1-\e)S(x,t)$ and $E \subset Z$, for all $\e \leq \e_1$ we have
\[
\begin{split}
\frac{|A(x,t)\cap S(x,(1-\e)t)|}{|S(x,t)|} &
\geq
\frac{|A(x,t)\cap (1-\e)S(x,t)|}{|S(x,t)|}
= \frac{|E\cap (1-\e)Z|}{|Z|}\\
& \geq
\frac{|E|}{|Z|} - \frac{|Z\setminus (1-\e)Z|}{|Z|} \geq c' - \left(1-(1-\e_1)^n\right) \geq \frac{c'}{2},
\end{split}
\]
provided $\e_1$ is sufficiently small.

Moreover, since on $A(x,t)$
\[
D^2 u(y)= \frac{1}{(\det T)^{2/n}} T^* D^2 v(Ty) T\ge \frac{c''}{(\det T)^{2/n}} T^*T, 
\]
using \eqref{eq:norms} we get
\[
\|D^2 u(y)\| \ge \frac{c''\|T^*T\|}{(\det T)^{2/n}} =\frac{c'' \|T^*\|\|T\|}{(\det T)^{2/n}}\qquad \forall\,y \in A(x,t),
\]
which proves \eqref{eq: hesssupermean}.
\end{proof}

Combining the two previous lemmas, we obtain that in every section we can find a uniform
fraction of points where the norm of the Hessian controls its average over the section.
As we will show below, Theorem \ref{main} is a direct consequence of this fact and a covering argument.

To simplify the notation, we use $M(x)$ to denote $M_{\U/2,3\U/4}(x)$
(see \eqref{eq:def max fct}).

\begin{lemma}
Let $u$ solve \eqref{MAnorm}. Then there exists two universal positive
constants $C_4$ and $C_5$ such that
\begin{equation}\label{levelsets}
|\{x\in  \U /2:\ M(x)\ge \gamma\}|\le C_4 |\{x\in 3\U/4\,:\,\ \|D^2 u(x)\|\ge C_5 \gamma\}|
\end{equation}
for every $\gamma>0$.
\end{lemma}

\begin{proof}
By the definition of $M$, we clearly have
\[
\{x\in  \U /2\, :\ M(x)\ge \gamma \} \subset
E:=\Big\{x \in \U/2\, :\ \mean{S(x,t_x)} |D^2 u|\ge \frac{\gamma}{2}
\textrm{ for some $t_x \in (0,\rho)$} \Big\}.
\]
By Proposition \ref{covering}, we can find a
sequence of points $x_k \in E$ such that
$\{S(x_k,t_{x_k})\}_{k\in \N}$ is a countable covering of $E$.
Since $x_k \in E$, by combining Lemmas \ref{hessmean} and \ref{hesssupermean}
we deduce that
\begin{equation}
\label{eq:hess k}
\|D^2 u(y)\| \ge C_1 C_3\mean{S(x_k,t_{x_k})} \|D^2 u\| \geq \frac{C_1 C_3 \gamma}2 \qquad \forall\, y \in A(x_k,t_{x_k}).
\end{equation}
Hence, applying
\eqref{A}, \eqref{eq:hess k}, and Proposition \ref{covering}, and
choosing $\e_2 := \min\{\e_0,\,\e_1\}$, we get
\[
\begin{split}
|\{x\in  \U /2:\ M(x) \ge \gamma \}|&\le \sum_{k\in \N}|S(x_k,t_{x_k})|\\
&\le \frac{1}{C_2}  \sum_{k\in \N}|A(x_k,t_{x_k})\cap S(x_k,(1-\e_2)t_{x_k})|\\
&\le\frac{1}{C_2} \sum_{k\in \N}\big|S(x_k,(1-\e_2)t_{x_k})\cap \{ x \in 3\U /4\,:\
\|D^2 u(x)\|\ge C_1 C_3 \gamma / 2\}|\\
&=\frac{1}{C_2}\int\limits_{3\U/4\cap \{ \|D^2 u(x)\|\ge C_1 C_3 \gamma / 2\}} \sum_{k\in \N}
\chi_{S(x_k,(1-\e_2)t_{x_k})}(y) \,\,dy\\
&\le\frac{K |\log \e_2|}{C_2}  \bigl|\{ x \in 3\U /4\,:\ \|D^2 u(x)\| \ge C_1 C_3 \gamma / 2\}\bigr|,
\end{split}
\]
proving the result.
\end{proof}

\begin{proof}[Proof of Theorem \ref{main}]
Combining \eqref{maximalineq} and \eqref{levelsets}, we obtain the existence of two positive universal constants
$c',c''$ such that
\begin{equation}\label{keyestimate}
\int\limits_{\U /2\cap \{\|D^2 u\|\ge \gamma\}} \|D^2 u\| \le
c'\gamma  |\{x\in 3\U/4:\ \|D^2 u(x)\|\ge c'' \gamma\}| \qquad \forall\,\gamma \geq \bar c,
\end{equation}
with $\bar c$ depending only on $\mean{3\U/4} \|D^2 u\|$ and $\rho$.
Observe that, since $u$ is normalized, both $\mean{3\U/4} \|D^2 u\|$ and $\rho$ are universal
(see the discussion at the beginning of this section), so $\bar c$ is universal as well.

Without loss of generality we can assume that $\bar c \geq 2$, so that
$\log(2+\gamma) \leq 2\log \gamma$ for all $\gamma \geq \bar c$
(this is done just for convenience, to simplify the computations below).
So
\[
\begin{split}
&\int_{\U /2} \|D^2 u\|\log^{k+1}(2+\|D^2u\|)\\
&\le \log(2+ \bar c) \int\limits_{\U /2\cap \{\|D^2 u\|\le \bar c\}} \|D^2 u\|
+\int\limits_{\U /2\cap \{\|D^2 u\|\ge \bar c\}} \|D^2 u\|\log^{k+1}(2+\|D^2 u\|)\\
&\le C(n) \bar c \log \bar c+2\int\limits_{\U /2\cap \{\|D^2 u\|\ge \bar c\}} \|D^2 u\|\log^{k+1}\|D^2 u\|.
\end{split}
\]
Hence, to prove the result, it suffices to control the last term in the right hand side.
We observe that such a term can be rewritten as
$$
2(k+1)\int\limits_{\U /2\cap \{\|D^2 u\|\ge \bar c\}}
\|D^2 u\|\int_{1}^{\|D^2u\|} \frac{\log^k(\gamma)}{\gamma}\,d\gamma,
$$
which is bounded by
$$
C'+ 2(k+1)\int\limits_{\U /2\cap \{\|D^2 u\|\ge \bar c\}}
\|D^2 u\|\int_{\bar c}^{\|D^2u\|} \frac{\log^k(\gamma)}{\gamma}\,d\gamma,
$$
with $C'=C'(k,\bar c)$. Now, by Fubini, the second term is equal to
$$
2(k+1) \int_{\bar c}^\infty
\frac{\log^k(\gamma)}{\gamma}\biggl(\int\limits_{\U /2\cap
\{\|D^2 u\|\ge \gamma\}} \|D^2 u\|\biggr) \,d\gamma,
$$
which by \eqref{keyestimate} is controlled by
$$
2(k+1) c' \int_{\bar c}^\infty
\log^k(\gamma) \,|\{x\in 3\U /4:\ \|D^2 u(x)\|\ge c'' \gamma\}\}| \,d\gamma.
$$
By the layer-cake representation formula, this last term is bounded by
$$
C''\int_{3\U /4} \|D^2 u\|\log^{k}(2+\|D^2u\|)
$$
for some $C''=C''(k,c',c'')$,
concluding the proof.
\end{proof}

\end{document}